\documentclass[12pt, reqno]{amsart}
\usepackage[utf8]{inputenc}
\usepackage{amsmath}
\usepackage{mathtools}
\usepackage{amsthm}
\usepackage{amssymb}
\usepackage{xcolor}
\usepackage{upgreek}
\usepackage{enumitem}
\usepackage[notcite, notref,color,final]{showkeys}
\definecolor{refkey}{rgb}{0,0,1}
\definecolor{labelkey}{rgb}{0,0,1}
\usepackage[margin=1.4in]{geometry}
\newtheorem{theorem}{Theorem}

\newtheorem{corollary}[theorem]{Corollary}

\newtheorem{question}[theorem]{Question}

\newcommand{\Sym}[1]{\mathfrak{S}_{#1}}
\newcommand{\Orb}{\mathrm{Orb}}
\newcommand{\GSym}[1]{\mathbb{C}\mathfrak{S}_{#1}}
\newcommand{\ZSym}[1]{Z(\mathbb{C}\mathfrak{S}_{#1})}

\newcommand{\precj}{\prec_j}
\newcommand{\doublehur}[1]{md_g(#1)}

\newcommand{\hur}[1]{m_g(#1)}
\newcommand{\hurprec}[2]{m^{#2}_g(#1)}
\newcommand{\sethurprec}[2]{M^{#2}_g(#1)}
\newcommand{\sethur}[1]{M_g(#1)}

\newcommand{\setdoublehur}[1]{MD_g(#1)}
\newcommand{\upth}{\text{th}}
\newcommand{\sethurgprimeprec}[2]{M^{#2}_{g'}(#1)}
\newcommand{\sethurgprime}[1]{M_{g'}(#1)}
\newcommand{\K}{\mathcal{K}}
\newcommand{\myatop}[2]{\genfrac{}{}{0pt}{2}{#1}{#2}}
\makeatletter
\def\smalloverbrace#1{\mathop{\vbox{\m@th\ialign{##\crcr\noalign{\kern3\p@}%
  \tiny\downbracefill\crcr\noalign{\kern3\p@\nointerlineskip}%
  $\hfil\displaystyle{#1}\hfil$\crcr}}}\limits}
\makeatother

 \def\@testdef #1#2#3{%
  \def\reserved@a{#3}\expandafter \ifx \csname #1@#2\endcsname
 \reserved@a  \else
 \typeout{^^Jlabel #2 changed:^^J%
\meaning\reserved@a^^J%
\expandafter\meaning\csname #1@#2\endcsname^^J}%
\@tempswatrue \fi}

\title{Centrality of star and monotone factorisations}

\author{Jesse Campion Loth}
	\address{School of Mathematics, University of Bristol and Heilbronn Institute for Mathematical Research, Bristol, UK}
	\email{j.campionloth@bris.ac.uk}
	\author{Amarpreet Rattan}
	\address{Simon Fraser University, Burnaby, BC, Canada}
	\email{rattan@sfu.ca}

\date{\today}
\def\@testdef #1#2#3{%
   \def\reserved@a{#3}\expandafter \ifx \csname #1@#2\endcsname
  \reserved@a  \else
 \typeout{^^Jlabel #2 changed:^^J%
 \meaning\reserved@a^^J%
 \expandafter\meaning\csname #1@#2\endcsname^^J}%
 \@tempswatrue \fi}

\begin{document}

\begin{abstract} A factorisation problem in the symmetric group is
	\emph{central} if conjugate permutations always have the same
	number of factorisations. We give the first fully combinatorial proof of
	the centrality of transitive star factorisations that is valid in all
	genera, which answers a natural question of Goulden and Jackson from
	2009.  We begin by showing that the set of star factorisations is
	equinumerous with a certain set of monotone factorisations, a new
	result. We give more than one proof of this, and, crucially, one of our
	proofs is bijective.  %
	As a corollary we obtain new formulae for some
	monotone double Hurwitz factorisations, and a new relation between Hurwitz and
	monotone Hurwitz factorisations.  We also generalise a theorem of Goulden and
	Jackson from 2009 that states that the transitive power of Jucys-Murphy
	elements are central.  Our theorem states that the transitive image of
	any symmetric function evaluated at Jucys-Murphy elements is central,
	which gives a transitive version of Jucys' original result from 1974.
\end{abstract}
\maketitle
\section{Introduction}\label{sec:intro}

Given a permutation $\gamma$, a factorisation of $\gamma$ is the writing of
$\gamma$ as the product of other permutations, subject to certain constraints.
Factorisation problems have received a lot of interest recently;  see the survey
of Goulden and Jackson \cite{goulsurvey}.  The central objects of the present manuscript
are monotone and star factorisations.  For both monotone and star
factorisations, the factors are restricted to the set of transpositions, but
both have additional constraints.  Both sets of factorisations
display a symmetry called \emph{centrality} | that conjugate elements have the same number of
factorisations | that is not apparent
from their definitions.  Our main results include new links between
these seemingly different factorisations, their centrality
properties and their connection with the Jucys-Murphy elements of the
symmetric groups.   Jucys-Murphy elements are famous for their properties when they are the arguments
of symmetric functions (Theorem \ref{thm:jucyscentral} below), and also for their role
in the representation theory of the symmetric group (see Okounkov and Vershik
\cite{ok:2}). We begin by establishing basic notation and recalling fundamental results in
Section \ref{sec:basics}.  The factorisations are then introduced 
in Section \ref{sec:factstrans} (monotone factorisations) and Section
\ref{sec:introstar} (star factorisations).  Our main results are stated in
Section \ref{sec:mainresults}.  Connections to geometry and consequences of our
main result are given in Section \ref{sec:connect}.  Proofs of the main results
appear in Sections \ref{sec:recurs} through \ref{sec:trans}.  Questions that arise from our main results and open problems are given in Section
\ref{sec:open}.

\subsection{Basics and notation}\label{sec:basics}
Let $n$ be a positive integer.   A \emph{partition} of $n$ is a weakly
decreasing list $\lambda = (\lambda_1, \lambda_2,
\ldots)$ of positive integers whose sum is $n$;  we denote this by $\lambda \vdash n$.  Each entry in
a partition is called a \emph{part}, and we denote by $\ell(\lambda)$ the number of parts.  
We use $\Sym{n}$ for the symmetric group on the set $[n] : =
\{1, 2, \ldots, n\}$ and $\GSym{n}$ for the group algebra of $\Sym{n}$,
which has centre denoted by $\ZSym{n}$.  Our convention is that we multiply
elements of $\Sym{n}$ from left to
right.
For a partition $\lambda \vdash n$, we let $C_\lambda$ be the conjugacy class of $\Sym{n}$
containing the permutations with cycle type $\lambda$.  Define the
\emph{conjugacy class sum} $K_\lambda \in \GSym{n}$ by
$$K_\lambda :=  \sum_{\sigma \in C_\lambda} \sigma.$$
It is well-known that each $K_\lambda$ is central in $\GSym{n}$ and that $\{K_\lambda : \lambda \vdash n\}$ is a linear basis for $\ZSym{n}$.  For $\sigma \in \Sym{n}$ and $g \in \GSym{n}$, we let $[\sigma] g$
be the coefficient of $\sigma$ in $g$.  Similarly, for any partition $\lambda \vdash n$ and any expression $K \in
\ZSym{n}$, we let $[K_\lambda] K$ be the coefficient of $K_\lambda$ in $K$.  We
write $c(\sigma)$ for the number of cycles in $\sigma$.

The \emph{Jucys-Murphy elements} are members of $\GSym{n}$ defined by $J_k = (1 \, k) + (2 \, k) +
\dots + (k-1 \, k)$ for $2 \leq k \leq n$, and $J_1 = 0$.  Their remarkable properties
were independently discovered by Jucys \cite{sympolys-jucys} and Murphy
\cite{murphy}.  For any $2 \leq i,j \leq n$, one can easily verify that $J_i
\notin \ZSym{n}$, but $J_i$ and $J_j$
commute.

We use notation consistent with Macdonald \cite{mac:1}.   We write symmetric functions in terms of 
indeterminates $x = (x_1, x_2, \ldots)$, and for a partition $\lambda \vdash k$, 
we write $h_\lambda(x)$, $e_\lambda(x)$ and $p_\lambda(x)$ for the
\emph{complete, elementary} and \emph{power sum} symmetric functions,
respectively.
We let $\Lambda^k$ be the symmetric functions of homogeneous degree $k$ and
$\Lambda := \oplus_{k \geq 0} \Lambda^k$ be the algebra of symmetric functions.

\sloppypar For $f \in \Lambda$, we let $f(\Upxi_n)$ be the evaluation defined by
$f(\Upxi_n) := f(J_1, \ldots, J_n, 0,
\ldots) = f(J_2, \ldots, J_n, 0, \ldots)$ (the latter equality is from $J_1 = 0$).
The following two fundamental results are due to Jucys \cite{sympolys-jucys} and
Murphy \cite{murphy, murphy1}.  
\begin{theorem}\label{thm:jm}
	For any integers $k,n \geq 1$,
	$$e_k(\Upxi_n) = \sum_{\myatop{\lambda \vdash n}{\ell(\lambda) = n-k}}
	K_\lambda.$$
	Whence $e_k(\Upxi_n)$ is central.
\end{theorem}
The next theorem states that symmetric function evaluations of Jucys-Murphy
elements are precisely $\ZSym{n}$.  Since $\{e_k : k \geq 1\}$ generate
$\Lambda$ as an algebra, the first part of the next theorem is a corollary of Theorem \ref{thm:jm}.
\begin{theorem}  \label{thm:jucyscentral}
    Let $n$ be a positive integer.  For any symmetric function $f$, we have $f(\Upxi_n) \in
    \ZSym{n}$.  Conversely, if $K \in \ZSym{n}$, there exists a symmetric
	function $f$ such that $K = f(\Upxi_n)$.  
\end{theorem}

\subsection{Monotone factorisations}  \label{sec:factstrans}

We focus on three well-known factorisation problems in $\Sym{n}$, two of
which we present in this section.  Throughout, when
a transposition $(a\, b)$ is written, we assume that $a < b$.  Let
$\omega \in \Sym{n}$ and $g$ be a nonnegative integer.   
A \emph{genus $g$ monotone double Hurwitz factorisation of
$\omega$} is a tuple $(\sigma, (a_1 \, b_1), (a_2 \, b_2), \ldots, (a_m \,
b_m))$ that satisfies $\sigma (a_1 \, b_1) (a_2 \, b_2) \cdots (a_m \, b_m) =
\omega$ and the following three conditions.
\begin{enumerate}
	\item[H0.] The permutation $\sigma$ is in $C_{(n)}$.
	\item[H1.]  The number of transpositions is $m =c(\omega) -1+2g$.
	\item[H2.]  The transpositions satisfy $b_1 \leq b_2 \leq \dots \leq
		b_m$ (\emph{i.e.} the transpositions are \emph{monotone}). 
\end{enumerate}
Let $\setdoublehur{\omega}$ be the set of such tuples and let
$\doublehur{\omega}$ be its cardinality.  For these factorisations, and other
factorisations below, we shall often omit tuple notation and simply write $\sigma (a_1 \,
b_1) (a_2 \, b_2) \cdots (a_m \, b_m) = \omega$ for a factorisation.

For the second problem, we are interested in tuples of transpositions $(a_1 \,
b_1) \allowbreak (a_2 \, b_2) \cdots (a_m \, b_m) = \omega$ satisfying the following two conditions.
\begin{enumerate}
	\item[H1'.] The number of transpositions is $m = n - c(\omega)  + 2g$.
	\item[H2'.] The transpositions satisfy H2.
\end{enumerate}
We let $\sethur{\omega}$ be the set of such tuples with
$\hur{\omega}$ its cardinality.\footnote{We give some more background to the
	sets $\setdoublehur{\omega}$ and $\sethur{\omega}$ and related problems at the end of Section
\ref{sec:introstar}.}  %
	Goulden, Guay-Paquet and Novak \cite{monotone_hurwitz_genus0} observed that 
\begin{equation}\label{eq:monotoneggpn}
	\hur{\omega} = [\omega] h_{n-c(\omega)+2g}(\Upxi_n).
\end{equation}
It follows from \eqref{eq:monotoneggpn} and Theorem \ref{thm:jm} that
\begin{equation}\label{eq:centsym}
	\doublehur{\omega} =[\omega] e_{n-1}(\Upxi_n) \cdot h_{c(\omega)-1+2g}(\Upxi_n)
				       = [\omega] J_2 \cdots J_n \cdot
				       h_{c(\omega)-1+2g}(\Upxi_n).
\end{equation}
From Theorem \ref{thm:jucyscentral}, the evaluations of symmetric functions 
on the right-hand sides of
\eqref{eq:monotoneggpn} and \eqref{eq:centsym} both lie in $\ZSym{n}$.   It
follows that if $\gamma$ is conjugate to $\omega$, so both lie in some conjugacy
class $C_\lambda$, then
\begin{equation*}
	\hur{\gamma}=\hur{\omega} = [\omega] h_{n - \ell(\lambda)+2g}(\Upxi_n),
\end{equation*} a
fact that is not obvious from the combinatorial description of the
factorisations.
We may therefore define $\hur{\lambda}:=
\hur{\omega}$ for the partition $\lambda$, and likewise define
$\doublehur{\lambda}$.  In light of this, 
we abuse terminology and say the enumerative functions $\hur{\cdot}$ and
$\doublehur{\cdot}$, or the associated factorisations, are
\emph{central}.\footnote{Equivalently, we have that $\hur{\cdot}$ and $\doublehur{\cdot}$ are class
functions on $\Sym{n}$.}

\subsection{Star factorisation and the transitivity operator}\label{sec:introstar}

We begin by recalling
the notion of a subset of $\Sym{n}$ acting transitively.  For a subset $\{\sigma_1, \dots, \sigma_t\} \subseteq \Sym{n}$,
let $\langle \sigma_1, \dots, \sigma_t \rangle$ be the subgroup of $\Sym{n}$
generated by the subset. 
Let $\Orb(\sigma_1, \dots, \sigma_t)$ be the set partition of $[n]$ describing
the orbits of the natural action on $[n]$ by $\langle \sigma_1, \ldots,
\sigma_t \rangle$.  That is, the elements $i,j \in [n]$
are in the same part of $\Orb(\sigma_1, \dots, \sigma_t)$ if and only if there
exists some $\sigma \in \langle \sigma_1, \dots, \sigma_t \rangle$ with
$\sigma(i) = j$.  We say $\{\sigma_1, \dots, \sigma_t\}$  \emph{acts
transitively on $[n]$} if $\Orb(\sigma_1, \ldots, \sigma_t) = \{[n]\}$, the set partition of $[n]$ with a single part.

For the third factorisation problem,  let $\omega \in \Sym{n}$ and $g$ be a
nonnegative integer.  A \emph{genus
$g$ transitive star factorisation} of $\omega$ is a tuple of
transpositions $(a_1\, n) (a_2\, n) \cdots (a_m \, n) = \omega$ 
that satisfies the following two conditions.
\begin{enumerate}
	\item[S1.] The number of transpositions is $m = n+c(\omega)-2+2g$.
	\item[S2.] The tuple of transpositions acts transitively on $[n]$.
\end{enumerate}
The terminology ``star" comes from the shape of the graph whose edges are given
by the allowable set of transpositions; each transposition must contain the
symbol $n$.
Because all factors contain the symbol $n$, condition S2 is equivalent to the
condition
\begin{enumerate}
	\item[S2'.] The tuple of transpositions contains $(i\, n)$ for each $1 \leq i \leq
		n-1$.\footnote{While this simpler condition for transitivity
	is more convenient for star factorisations, we presented S2 with the more
	general transitivity condition that applies to other factorisation types for historical context.  Transitivity
	conditions are often present in factorisation problems in $\Sym{n}$;  see Goulden and
Jackson \cite{goulsurvey}.} 
\end{enumerate}
 Let $A_g(\omega)$ be the set of such factorisations
and $a_g(\omega)$ be its cardinality.
Without the transitivity constraint S2, the number of genus $g$ star
factorisations of $\omega$ is evidently $[\omega] J_n^{m}$.  
To account for transitivity,  we introduce a map, 
denoted by $T_n$.  It is defined on products of permutations by
$$
    T_n(\tau_1 \tau_2 \dots \tau_t) = \begin{cases}
		\tau_1 \tau_2 \dots \tau_t, & \text{if } \{\tau_1, \ldots, \tau_t\}
		\text{ acts transitively on } [n], \\
		0, & \text{otherwise}. 
    \end{cases}
$$
We extend $T_n$ linearly to any finite polynomial expression of the form
$\sum_i a_i C^{t_i}_{21^{n-2}}$.
In particular, for any power of a Jucys-Murphy element $J_i^j$ and symmetric
function $f$, the expressions $T_n(J_i^j)$ and $T_n(f(\Xi_n))$ are well-defined.
Using the same notation, we define a linear operator $T_n : \Lambda \rightarrow \GSym{n}$ by
$T_n(f) : = T_n(f(\Xi_n))$.   The operator $T_n$ is therefore by definition
linear, but it clearly not an algebra homomorphism, i.e. $T_n(fg) \neq
T_n(f)T_n(g)$.  For example, we find
$T_n(e_{n-1} e_1)$ is not 0, but $T_n(e_{n-1}) T_n(e_1) = 0$.  We continue to
write $T_n(f(\Xi_n))$ even with this understanding to be explicit.

The connection between transitive star factorisations and $T_n$ is
simple;  for $m=n+c(\omega) - 2 + 2g$, we have
\begin{equation}\label{eq:starjm}
	a_g(\omega) = [\omega] T_n(J_n^{m}).
\end{equation}  
The consideration of $T_n(J^m_n)$ was introduced by Goulden and Jackson \cite{goul:8}, who
referred to it as the \emph{transitive power of
the Jucys-Murphy element $J_n$}.  It is easy to verify that $J_n^{m}$ 
is not central in general.  From \eqref{eq:starjm}, the function $a_g(\cdot)$ is central
(\emph{i.e.} $a_g(\omega) = a_g(\gamma)$ whenever $\omega$ and $\gamma$ are conjugate)
if and only if $T_n(J_n^{m})$ is central in the group algebra.   The following
theorem is also found in \cite{goul:8}.
\begin{theorem}\label{thm:transpower}
	For any positive integer $n$ and nonnegative integer $t$, the transitive power
	$T_n(J_n^{t})$ is central;  \emph{i.e.} $T_n(J_n^{t}) \in
	\ZSym{n}$.  Equivalently, for any nonnegative $g$, the function $a_g(\cdot)$
	is central.
\end{theorem}
One can easily observe that 
\begin{equation}\label{eq:transpower}
	T_n(J_n^t) = T_n(p_t(\Upxi_n)), 
\end{equation}
but clearly
$J_n^t \neq p_t(\Upxi_n)$, as seen in the following example where $n=t=4$:
\begin{align*}
	J_4^4 &= 15e + 8\left( (2\, 3\, 4) + (2\, 4\, 3) + (1\, 2\, 4) + (1\, 3\,
	4) + (1\, 4\, 2) + (1\, 4\, 3) \right)\\
	      & \phantom{+++}+ 3 \left(
			(1\, 2\, 3) + (1\, 3\, 2) \right) + 4\left( (1\, 2) (3\,
				4) + (1\, 4)(2\, 3) \right)\notin \ZSym{n},\\
	p_4(\Upxi_4) &= 22K_{1111} + 8K_{31} + 4K_{22},\\
	T_4(J_4^4) &= T_4(p_4(\Upxi_4)) = 3K_{31} + 4K_{22}.
\end{align*}
That $a_g(\cdot)$ is central for any $g$ is
a surprising fact given the asymmetric role of the symbol $n$ in star
factorisations.  We may therefore write $a_g(\lambda)$ to be the number of genus
$g$ transitive star factorisations of any fixed permutation in $C_\lambda$.   We
emphasise that the transitivity constraint S2 is crucial for the centrality of
$a_g(\cdot)$; this follows from $T_n(J_n^m)$ being central while $J_n^m$ is not.
More concretely, letting $\omega = (1\, 2)(3)$ and $\gamma = (1)(2\, 3)$, one can easily confirm
that $a_0(\omega) = a_0(\gamma) =2$, but the number of genus 0 star factorisations that
are not constrained by S2 of $\omega$ and $\gamma$ are 2 and 3, respectively.
Since transitivity is essential to centrality for star
factorisations, in what follows transitivity will be assumed even when we simply write \emph{star factorisations}, unless specifically noted otherwise. 

We note that if $\omega$ and $\gamma$ are conjugate and have no fixed points,
then
\begin{equation}\label{eq:notfixed}
	[\omega] T_n(J_n^\ell) = [\omega] J_n^\ell = [\omega] p_\ell(\Upxi_n) =
	[\gamma] p_\ell(\Upxi_n) = [\gamma] J_n^\ell = [\gamma] T_n(J_n^\ell).
\end{equation}
The middle equality holds by Theorem \ref{thm:jucyscentral}.  Thus the main content of Theorem \ref{thm:transpower} is that
\eqref{eq:notfixed} also holds when $\omega$ and $\gamma$ are conjugate and
have fixed points.  The work on star factorisations was initiated by Pak \cite{pak1999reduced},
where the author found a formula for the case $g=0$ when all cycles of the
target permutation $\omega$
have some fixed length $k$, except the symbol $n$ is a fixed point.  Irving and 
	the second author
	\cite{rattanirving:1} then found a general formula for $a_0(\omega)$ for
	any $\omega \in \Sym{n}$, and it was in \cite{rattanirving:1} that the centrality phenomenon was
	first observed.  Goulden and Jackson then found a formula for
	$a_g(\omega)$ for all $g \geq 0$ and all $\omega \in \Sym{n}$, confirming that the
	centrality phenomenon holds in general.  Goulden and Jackson's results in \cite{goul:8} are, in fact, much stronger than Theorem \ref{thm:transpower} as they give full
enumerative formulae for $a_g(\omega)$.
However, the techniques of Goulden and Jackson and of Irving and 
the second author
do not expose why $a_g(\cdot)$ is
central \emph{combinatorially}:  centrality of $a_g(\omega)$ is obtained as a
corollary of their formulae by observing that the formulae depend only on
the conjugacy class of $\omega$.  Goulden and Jackson \cite[Section
4]{goul:8} thus explicitly pose the natural problem of finding a combinatorial
explanation of the centrality of $a_g(\cdot)$ that works in all genera.  This problem was one of the main motivations for the present article, and
its resolution one of our main results (see Section \ref{sec:mainresults}).
Further related work in this area was done by F\'eray \cite{feray4}, who gave 
algebraic arguments that allowed him to greatly simplify Goulden and Jackson's
formulae;  see \eqref{eq:feray} below. Tenner \cite{tenner, tenner2} gave two different
intriguing combinatorial constructions proving the centrality of $a_g(\cdot)$,
but they are limited to the case $g=0$.

\subsection{Main results}\label{sec:mainresults}
We now give the main results of this paper, which tie together the objects
introduced above.   Our first new result connects monotone double Hurwitz
factorisations to those
enumerating star factorisations.
\begin{theorem}  \label{thm:hurwitzrelation}
	Let $\omega  \in \Sym{n}$.  Then for each nonnegative integer $g$, we have
    $$a_g(\omega) = \doublehur{\omega}.$$
\end{theorem}
In Section \ref{sec:recurs} we give a simple recursive proof of this result.
This proof will be independent of the combinatorial proofs found in
Sections \ref{sec:combmon} and \ref{sec:combstar}.

It was explained in Sections \ref{sec:factstrans} and \ref{sec:introstar} that $a_g(\cdot)$, $\hur{\cdot}$ and
$\doublehur{\cdot}$ are known to be central from past results. Our next results
are that these can all be shown \emph{combinatorially}, and our arguments apply
to all $g \geq 0$, thus resolving the problem of finding a combinatorial proof 
of the centrality of star factorisations posed by Goulden and Jackson.  Our
resolution to their problem takes the following path.  Let $\omega$ and $\gamma$
be conjugate permutations in $\Sym{n}$ and $g$ be a nonnegative integer.  We first give 
bijective proofs of the centrality $\hur{\cdot}$ and $\doublehur{\cdot}$
in Section \ref{sec:combmon}.  In particular, in Corollary \ref{cor:mondouble},
we give a bijection $\Theta : \setdoublehur{\omega}
\rightarrow \setdoublehur{\gamma}$.  Next, in Theorem \ref{thm:bijectiongamma}, we give
a bijection $\Gamma : A_g(\omega) \rightarrow \setdoublehur{\omega}$ proving Theorem \ref{thm:hurwitzrelation}. 
Finally, composing the bijections 
\begin{equation}\label{eq:centstar}
	A_g(\omega) \overset{\Gamma}{\longrightarrow} \setdoublehur{\omega}
	\overset{\Theta}{\longrightarrow} \setdoublehur{\gamma}
	\overset{\Gamma^{-1}}{\longrightarrow} A_g(\gamma)
\end{equation}
gives a combinatorial proof of the centrality of star
factorisations.

It follows from Theorems \ref{thm:jucyscentral} and \ref{thm:transpower}  that
$T_n(J_n^t)$ is expressible as a symmetric function of Jucys-Murphy elements.
The following corollary, which immediately follows from Theorem
\ref{thm:hurwitzrelation} and \eqref{eq:centsym} and \eqref{eq:transpower}, gives such an expression.
\begin{corollary}  \label{cor:transpowerexpression}
    For any $n \geq 2$ and $k \geq 0$, we have
    $$ T_n(p_{n-1+k}(\Upxi_n)) = T_n(J_n^{n-1+k})  = J_2 \cdots J_n \cdot
    h_k(\Upxi_n) = e_{n-1}(\Upxi_n) \cdot h_k(\Upxi_n).$$
\end{corollary}
Our final result, which is proved in Section \ref{sec:trans}, is along the lines of Corollary \ref{cor:transpowerexpression},
and gives a result analogous to the first part of Theorem \ref{thm:jucyscentral}
for transitive images.  In particular, it also gives another indirect proof of the centrality of star factorisations.
\begin{theorem} \label{thm:transsym}
    For any symmetric function $f$, we have $T_n(f(\Upxi_n)) \in Z(\GSym{n})$.
\end{theorem}

\subsection{Connections to geometry and consequences of the main
results}\label{sec:connect}

The set $\setdoublehur{\omega}$ is a special case of the following.  For
partitions $\alpha, \beta \vdash n$, consider the set of tuples $(\sigma, (a_1 \, b_1),
(a_2 \, b_2), \ldots, (a_m \, b_m))$ whose product is in $C_\beta$ that
satisfies the following conditions:
\begin{enumerate}
	\item[H0''.]  $\sigma \in C_\alpha$; 
	\item[H1''.]  $m=\ell(\alpha) + \ell(\beta) -2 + 2g$;
	\item[H2''.]  $b_1 \leq b_2 \leq \cdots \leq b_m$ (monotonicity
		condition);
	\item[H3''.] the factor $\sigma$ along with the transpositions act transitively on $[n]$.
\end{enumerate}
These conditions define transitive monotone double Hurwitz factorisations in
general. The set $\setdoublehur{\omega}$ is the special case where $\alpha =
(n)$, but is different in another way:  tuples in $\setdoublehur{\omega}$ have a
product $\omega$ that is fixed and not one that merely belongs to a specified
conjugacy class.  Also, since $\alpha = (n)$ in $\setdoublehur{\omega}$, the
factorisations are automatically transitive, explaining why
condition H3'' is omitted in the definition of $\setdoublehur{\omega}$.   The set
$\sethur{\omega}$, as presented here, does not have a transitivity constraint on
it. That transitivity is omitted as an additional constraint for
$\sethur{\omega}$ allows \eqref{eq:monotoneggpn} to hold. The articles
\cite{conv-goulden-guaypaq-novak, poly_monotone, monotone-hciz} consider monotone factorisations in depth.

The set of \emph{double Hurwitz factorisations} $H^g_{\alpha, \beta}$
are defined similarly to monotone double Hurwitz factorisations,
but the monotonicity condition H2'' is relaxed.  Let $|\mathrm{Aut}(\alpha)| =
\prod_i a_i!$, where $a_i$ are the number of parts of $\alpha$ equal to $i$.
The double Hurwitz \emph{numbers} are given by
$\frac{|\mathrm{Aut}(\alpha)||\mathrm{Aut}(\beta)|}{n!} |H^g_{\alpha, \beta}|$, and
they count the number of degree $n$ connected branched covers of the sphere by a
Riemannian surface of genus $g$ with $m + 2 = \ell(\alpha) + \ell(\beta) + 2g$ branch points, with branching of type $\alpha$
and $\beta$ over $0$ and $\infty$, respectively, and simple branching at
the remaining $m$ points (corresponding to the $m$ transpositions);  see Goulden,
Jackson and Vakil \cite[Proposition 1.1]{goulden2005towards}, who attribute
the discovery of the connection between factorisations and geometry to Hurwitz
\cite{hurwitz2}.  These numbers have a rich theory, with links to
intersection theory, tropical geometry, lattice point enumerations, and have
been the focus of much recent study; see for example \cite{borot2023double}.
The combinatorics of these numbers have been well studied in
genus zero, but are less well understood in higher genera.  Since monotone double Hurwitz
factorisations are a strict subset of double Hurwitz factorisations, monotone
double Hurwitz factorisations count a subset of these coverings of the
sphere.  They also give a combinatorial interpretation of the genus expansion of
the HCIZ integral,  and they are also related to free energy \cite{monotone-hciz}.

The Jucys-Murphy elements arise in the representation theory of the
symmetric group.  They generate the \emph{Gelfand-Tsetlin algebra}, a fact
that can
be used to inductively define all the representations of the symmetric group;
see \cite{ok:2}.   There has been a lot of work on expressions of symmetric functions evaluated at Jucys-Murphy elements \cite{feray2012,
classlassalle}, including the power sums, whose evaluations have links with
vertex operators in mathematical physics \cite{lascoux2004vertex}.  

For double Hurwitz factorisations, and their monotone version, the presence of
the transitivity constraint comes naturally from the geometry of branched covers
(the covers are connected), but the constraint of transitivity on star factorisations
seems ad-hoc;  the main justification for considering the constraint is that it
makes star factorisations, and thus transitive powers of Jucys-Murphy elements, central.  Theorem \ref{thm:hurwitzrelation} and its combinatorial proof
in Section \ref{sec:combstar}, however, illuminates why the transitivity
constraint on star factorisations may be natural, since it relates them to 
monotone double Hurwitz factorisations when $\alpha = (n)$, where the transitivity
constraint is forced.  These ideas link the importance of the transitivity
constraint in problems coming from the different areas of algebraic geometry and
representation theory, and give further evidence as to why transitivity is an
important constraint.

\subsection{Hurwitz moves} We introduce a fundamental tool in our analysis below:  \emph{Hurwitz
moves}.   We give the necessary details here, but the reader is referred to
\cite{adinroichman} for a different application.  For permutations $\gamma$
and $\omega$, we set $\gamma^\omega := \omega \gamma
\omega^{-1}$. A Hurwitz move takes as input a
factorisation and a pair $\tau \sigma$ of adjacent transpositions,
and it outputs a factorisation of the same length with only the factors $\tau$
and $\sigma$ changed.  There are two
kinds of Hurwitz moves. 
\begin{itemize}
	\item The \emph{rightward Hurwitz move} (RHM) takes $\tau \sigma$ to
		$\sigma^\tau \tau$.  
	\item The \emph{leftward Hurwitz move} (LHM) takes $\tau \sigma$ to $\sigma
		\tau^\sigma$.
\end{itemize}
In the output of an RHM we interpret $\sigma^\tau \tau $ as the
product of the two transpositions $\sigma^\tau$ (interpreted as one
transposition) and $\tau$;  we do likewise for an LHM.  Thus the number of
factors in a factorisation after applying a Hurwitz move is indeed preserved.
Observe that the two moves are inverses of each other.  Next,  if $\tau$ and
$\sigma$ commute (\emph{i.e.} they have disjoint support) than each Hurwitz move
reduces to simply swapping $\tau$ and $\sigma$ in the factorisation.   Finally,
both Hurwitz moves preserve the product of the moved factors, so both moves
preserve the product of the factorisation.

\section{A first proof of Theorem \ref{thm:hurwitzrelation}} \label{sec:recurs}

In this section we give a proof of Theorem \ref{thm:hurwitzrelation} by showing
both sides of the equation satisfy the same recurrence.  Though this proof is
short, we allow
ourselves the use of the centrality of monotone double Hurwitz factorisations.
The longer bijective proofs in the upcoming sections do not assume
centrality.  We use a recurrence for
star factorisations found in \cite{goul:8}.  In that recurrence, when finding
factorisations of $\omega$, it is necessary to record the length of the
cycle containing the distinguished symbol $n$.  To that end, we
introduce the following notation.  For positive integers
$k$ and $i$ and $\gamma \vdash k$, we let $\gamma \cup i$ be the
partition of $k+i$ consisting of all the parts of $\gamma$ plus an additional
part equal to $i$.  Let $\K^{(i)}_\gamma$ be the set of $\omega \in \Sym{n}$ such
that the symbol $n$ is in a cycle of length $i$, while the remaining cycle
lengths are given by $\gamma$.  Thus $\omega \in C_{\gamma \cup i}$ if $\omega
\in \K^{(i)}_\gamma$.

Fix positive integers $i \leq n$, and let $\alpha$ be a partition of $n-i$.  Let
$a_g(i, \alpha)$ denote the number of genus $g$ transitive star factorisations of any
permutation $\omega \in \K^{(i)}_\alpha$.  The quantity $a_g(i, \alpha)$ is well
defined, even without the assumption of the centrality of star factorisations.  Goulden and Jackson \cite[Theorem 2.1]{goul:8} found the following recurrence that specifically applies to $a_g(i, \alpha)$:
\begin{equation}\label{eq:recur}
	a_g(i, \alpha) = a_g(i-1, \alpha) + \sum_{t=1}^{\ell(\alpha)} \alpha_t a_g(i + \alpha_t, \alpha \setminus \alpha_t) + \sum_{t=1}^{i-1} a_{g-1}(i-t, \alpha \cup t).
\end{equation}
Let $\doublehur{i, \alpha}$ be the number of genus $g$ monotone double Hurwitz factorisations
of any $\omega \in \K^{(i)}_\alpha$.  The quantity $\doublehur{i, \alpha}$ is well defined by the
centrality of these factorisations.
We shall see that the numbers
$\doublehur{i, \alpha}$ satisfy the same recurrence \eqref{eq:recur}, giving our first proof of Theorem \ref{thm:hurwitzrelation}.

Now let $\alpha \cup i \vdash n$ and $\omega \in \K^{(i)}_\alpha$.  Let  $\sigma \tau_1 \tau_2 \dots \tau_m = \omega$ be a factorisation in
$\setdoublehur{\omega}$.  We split our analysis into two cases.  We use a classic
\emph{join-cut analysis}, which was also used by Goulden and Jackson to prove
\eqref{eq:recur}.  This involves the following elementary analysis on permutation
products $\nu \tau$, where $\tau = (i\, j)$ is a
transposition.  In the product $\nu \tau$, there are two possibilities for
the relative position of $i$ and $j$ given their relative position in $\nu$:  if $i$ and $j$ are in different cycles of $\nu$, then in $\nu \tau$
the symbols $i$ and $j$ are in the same cycle ($\tau$ is a \emph{join} for
$\nu$);  otherwise, the symbols $i$ and $j$ are in different cycles of
$\nu \tau$ ($\tau$ is a \emph{cut} for $\nu$).  

Case 1: $\tau_m = (j \, n)$ for some $j<n$.  Then $\sigma \tau_1 \tau_2 \dots
\tau_{m-1} = \omega (j \, n)$ is a monotone Hurwitz factorisation of $\beta:=\omega (j \, n)$.  We then have two cases based on the positions of $j$ and $n$ in the cycles of $\omega$.

Case 1a: $j$ and $n$ are in the same cycle of $\omega$.  Then $(j \, n)$ is a
cut for $\omega$, and $j$ and $n$ are in separate cycles $\beta$.  Then
for some $t < i$ the symbol $n$ is in a new cycle of length $i-t$ and $j$ is in a new
cycle of length $t$ in $\beta$, so $\beta \in \K^{(i-t)}_{\alpha \cup t}$.  Here
$t$ is obtained as the smallest natural number such that $\omega^t(j) = n$.  A quick computation gives that $\sigma \tau_1 \dots
\tau_{m-1}$ is a $g-1$ genus factorisation of $\beta$.  Since there are $i-1$ other symbols in the
same cycle as $n$ in $\omega$, we obtain each possible value of $t=1,2,\dots,i-1$ in this way.  The contribution from this case is therefore $\sum_{t=1}^{i-1} md_{g-1}(i-t, \alpha \cup t)$.

Case 1b: $j$ and $n$ are in different cycles of $\omega$.  Then $(j \, n)$ is a
join for $\omega$, and two cycles of $\omega$ are merged into a new cycle
containing $n$ in $\beta$.  This means that if $j$ is in a cycle of length
$\alpha_t$, then $\beta \in \K^{(\alpha_t + i)}_{\alpha \setminus \alpha_t}$ and
$\sigma \tau_1 \dots \tau_{m-1}$ is a factorisation of $\beta$ in
$\setdoublehur{\beta}$.  For each $t=1,\dots,\ell(\alpha)$, there are
$\alpha_t$ choices of $j$ that are in a cycle of length $\alpha_t$.  The
contribution from this case is therefore $\sum_{t=1}^{\ell(\alpha)} \alpha_t
\doublehur{i + \alpha_t, \alpha \setminus \alpha_t}$.

Case 2: $\tau_m = (j \, k)$ with $j < k < n$.  Since the factorisation $\sigma
\tau_1 \cdots \tau_m = \omega$ is monotone, no transposition contains
the symbol $n$.  It follows that $\sigma^{-1}(n) =
\omega^{-1}(n)$.  Set $t:=\omega^{-1}(n)$.  Observe that $\sigma':= (t\, n) \sigma$ is a permutation with a single $(n-1)$-cycle and a fixed point $n$.  Then 
\begin{equation}\label{eq:case2jesse}
	\sigma' \tau_1 \tau_2 \dots \tau_m = \beta, 
\end{equation}
where $\beta:=(t\, n) \omega$.  Note that $n$ is a fixed point of $\beta$.  It follows that we
can view the factorisation \eqref{eq:case2jesse} as a monotone double Hurwitz
factorisation in  $\Sym{n-1}$.  The construction in \eqref{eq:case2jesse} gives a bijection to all factorisations in $\setdoublehur{\beta}$.
Clearly $\beta$ has cycle type $\alpha \cup i-1$.   By the centrality of
	monotone factorisations, the set $\setdoublehur{\beta}$ is counted by
	$\doublehur{\alpha \cup i-1} = \doublehur{i-1, \alpha}$.

\sloppy Putting these cases together gives the following recurrence, which is the same as for transitive star factorisations:
\begin{equation*}
    \doublehur{i, \alpha} = \doublehur{i-1, \alpha} + \sum_{t=1}^{\ell(\alpha)} \alpha_t
    \doublehur{i + \alpha_t, \alpha \setminus \alpha_t} + \sum_{t=1}^{i-1} md_{g-1}(i-t, \alpha \cup t).
\end{equation*}
\sloppy The initial conditions are also the same for the two problems: $a_0(1,
\epsilon) = md_0(1, \epsilon) = 1$, where $\epsilon$ indicates the empty
partition.  We therefore see for each permutation $\omega \in \K^{(i)}_\alpha$
has $a_g(\omega) = md_g(\omega)$, and the result follows.

\section{A combinatorial proof of centrality of monotone factorisations}\label{sec:combmon}

In this section we give combinatorial proofs of the centrality of both types of monotone
Hurwitz factorisations, and we begin with $M_g(\omega)$.   Let $\prec$ be the ordering on $[n]$ given by $i_1 \prec i_2 \prec \cdots \prec
i_n$ and  let $\omega \in \Sym{n}$.  A \emph{monotone factorisation of $\omega$
relative to $\prec$} is a factorisation where condition H2 is replaced by the
condition that each transposition satisfies $a_i \prec b_i$ and $b_1 \preceq b_2
\preceq \cdots \preceq b_n$.  Let $\sethurprec{\omega}{\prec}$
be the set of such factorisations and $\hurprec{\omega}{\prec}$ its
cardinality.  Thus for the natural order $<$ on $[n]$, we have
$\sethur{\omega} =\sethurprec{\omega}{<}$ by definition.

Now let $\precj$ be the ordering obtained from $\prec$, but $i_j$ and $i_{j+1}$
are swapped.  We have the following theorem, whose proof is fundamental to this
manuscript.  The promised bijective proofs of the centrality of monotone Hurwitz
and double Hurwitz factorisations are given in Corollaries \ref{cor:order} and
\ref{cor:mondouble}.

\begin{theorem}\label{thm:bijectioninc}
	Let $\omega \in \Sym{n}$ and $g$ a nonnegative integer.  Then for every
	integer $1 \leq j \leq n-1$, there exists a bijection $\Lambda_j : \sethurprec{\omega}{\prec}
	\rightarrow \sethurprec{\omega}{\precj}$.
\end{theorem}
\begin{proof}
	Throughout, when a transposition $(a\, b)$ is written, it is assumed
	that $a$ is less than $b$ in the order being considered.  The ordering
	at the beginning of the proof will be $\prec$, and this will remain so until
	noted.
	We define the function $\Lambda_j$ in two stages.   Let $f \in
	\sethurprec{\omega}{\prec}$.

\textbf{Stage 1:} The factorisation $f$ has a maximal contiguous
	substring of transpositions (possibly empty, in which case $\Lambda_j$ fixes
	$f$) of the form
\begin{equation}\label{eq:substring}
	\cdots \underbrace{( \cdot \, i_j) \cdots (\cdot \, i_j)}_{\text{string }1} \underbrace{(\cdot \, i_{j+1}) \cdots (\cdot \,
	i_{j+1})}_{\text{string }2} \cdots.
\end{equation}
String $1$ contains all transpositions of the form $(a\, i_j)$ where $a \prec i_j$ and
string $2$ contains all transpositions of the form $(a\, i_{j+1})$, where
$a \prec i_{j+1}$.
The transpositions to the left of string 1 contain only symbols less than
$i_j$ under $\prec$, while all the transpositions to the right of string $2$
have the form $(a\, b)$ where $i_{j+1} \prec b$.  The function $\Lambda_j$ will
produce a factorisation of $\omega$, and it will leave all transpositions
outside of the strings 1 and 2 fixed, so $f$ and $\Lambda_j(f)$ will differ on
only the substring \eqref{eq:substring}.  Note that $f$ may contain some
transpositions of the form $(i_j\, i_{j+1})$, and these will only be in string
$2$.  If either string 1 or string 2 is empty, then string 1' = string 2 (former
case) or string 2' = string 1 (latter case) in \eqref{eq:intermed} in the output
of Stage 1 below.  We assume otherwise in what remains of Stage 1.

Beginning with the rightmost transposition in string $1$,
the function $\Lambda_j$ will apply Hurwitz moves to adjacent
transpositions of the form $(a\, i_j)$ and $(b\, i_{j+1})$ according to
the following cases:  1) $a \neq b$ and $b \neq i_j$; 2) $a=b$; and 3) $b=i_j$.  Note that $a \neq i_{j+1}$ because $a \prec i_j \prec i_{j+1}$. 
\begin{enumerate}[leftmargin=9ex]
	\item[Case 1:] $a \neq b$ and $b \neq i_j$.  In this case the two
		transpositions have disjoint support and commute, so we may apply either a LHM or RHM.  The output of the move is the same in either case, and it is $(b\, i_{j+1})
		(a\, i_j)$.
	\item[Case 2:] $a=b$.  In this case we apply a RHM, and the output of
		the move is $(i_j\, i_{j+1}) (a\, i_j)$.
	\item[Case 3:] $b=i_j$.  In this case we also apply a RHM, and the output of the move is
		$(a\, i_{j+1}) (a\, i_j)$.
\end{enumerate}
In all Cases 1-3, the two transpositions created by the operations are
each, individually, consistent with the order $\prec$.  For example, if Case 1 is
applied, both transpositions $(b\, i_{j+1})$ and $(a\, i_j)$ are consistent with
$\prec$, \emph{i.e.} $b \prec i_{j+1}$ and $a \prec i_j$.  This observation includes the transposition $(i_j\, i_{j+1})$ produced
in Case 2.

After an application of one of the moves above, in all cases the transposition
on the right is $(a\, i_j)$,
while the transposition to the left contains $i_{j+1}$.  We continue to move
the transposition $(a\, i_{j})$ to the right in this manner 
until there are no transpositions of the form $(b \, i_{j+1})$ to the
right of it.  We apply this process again for the second
rightmost transposition in string $1$ in the original factorisation $f$, and
then repeat.  In particular, a transposition of the form $(i_j\, i_{j+1})$ is never moved to the right.

This ends Stage 1.  Notice that all the transpositions originally in string $1$ are
preserved, in order, and to the right of all transpositions originally in string $2$ of $f$,
all of which may have changed, but all have the property that they contain $i_{j+1}$.
As Hurwitz moves preserve the product of the factorisation, the resulting object is
a factorisation of $\omega$.  Thus, at this point, we have a factorisation of
$\omega$ of the form
\begin{equation}\label{eq:intermed}
	\cdots \underbrace{( \cdot \, i_{j+1}) \cdots (\cdot \, i_{j+1})}_{\text{string }1'} \underbrace{(\cdot \, i_{j}) \cdots (\cdot \,
	i_{j})}_{\text{string }2'} \cdots,
\end{equation}
where string $2'$ is precisely string $1$.  This new factorisation may not, however, lie in
$\sethurprec{\omega}{\precj}$ because of the presence of transpositions of
the form $(i_j\, i_{j+1})$ in string $1'$ (in $\precj$ we have $i_{j+1} \precj
i_j$).  These transpositions may have been in the original factorisation $f$ or
created in Case 2.  Transpositions of the form $(i_j\, i_{j+1})$ are only in
string $1'$.  We deal with these transpositions in Stage 2 if they exist,
otherwise the algorithm terminates here.  But we first note that we can recover the original
factorisation by starting with the leftmost factor in string $2'$ of
\eqref{eq:intermed} and reversing the process in Stage 1 by applying LHMs.

\textbf{Stage 2:}  Let the position of the rightmost transposition of the form $(i_j\, i_{j+1})$ in
string $1'$ be $r$.  To its right is a substring of transpositions in string $1'$ consisting of transpositions of the form $(a\, i_{j+1})$ where $a \neq i_j$;  let the transpositions in this
string be $(a_k\, i_{j+1})$ for $1 \leq k \leq t$ (note that $t=0$ is possible
in which case the substring is empty).  That is,
the factorisation at the end of Stage 1 has the form
\begin{equation}\label{eq:premove}
	\cdots \underbrace{( \cdot \, i_{j+1}) \cdots (i_j\, i_{j+1}) \cdots
		\smalloverbrace{(i_j\, i_{j+1})}^{\mathclap{r^\upth \text{ transposition}}} (a_1\, i_{j+1}) \cdots (a_t\, i_{j+1})}_{\text{string }1'} \underbrace{(\cdot \, i_{j}) \cdots (\cdot \,
	i_{j})}_{\text{string }2'} \cdots.
\end{equation}
Move the transposition $(i_j\, i_{j+1})$ in position $r$ to after the
transposition $(a_t\, i_{j+1})$, and for each transposition $(a_k\, i_{j+1})$,
replace it with $(a_k\, i_{j})$.   The output of applying this
procedure to \eqref{eq:premove} is
\begin{equation}\label{eq:stage2}
	\cdots ( \cdot \, i_{j+1}) \cdots (i_j\, i_{j+1}) \cdots
	{\smalloverbrace{(\phantom{i_j\,
	i_{j+1}})}^{\mathclap{\myatop{r^{\scriptscriptstyle \upth}
{\scriptstyle \text{ transposition}}}{\text{moved}}}}}
	(a_1\, i_{j}) \cdots (a_t\,
	i_{j})\smalloverbrace{(i_{j+1}\, i_{j})}^{\myatop{{\scriptstyle \text{to}}}{\text{here}}} \underbrace{(\cdot \, i_{j}) \cdots (\cdot \,
	i_{j})}_{\text{string 2'}} \cdots.
\end{equation}
We write the new transposition as $(i_{j+1}\, i_j)$  to indicate that it is now
consistent with the order $\precj$.
We leave it to the reader to check that this move preserves the product of the
factorisation (moving the transposition as stated can be seen as a sequence of RHMs).
Note that if there is only one transposition $(i_j\, i_{j+1})$ in
\eqref{eq:premove}, then in \eqref{eq:stage2} after $(i_j\, i_{j+1})$ has been
moved, the transpositions $(a_1\, i_j), \ldots,  (a_t\, i_j)$ are the
only transpositions to the left of the newly positioned $(i_{j+1}\, i_{j})$
that contain the symbol $i_j$. Thus these transpositions can be identified, and this procedure can be reversed. 

We repeat this process on the remaining transpositions of the form $(i_j\,
i_{j+1})$ in string 1' until they have each been moved to the right of all transpositions of the form 
$(a\, i_{j+1})$, where $a$ is less than $i_{j+1}$ in both $\prec$
and $\precj$.  
This ends Stage 2 of the process and completes the description of $\Lambda_j$.

Stage 2 is reversible because of the following reasoning.  We find the leftmost
amongst all the transpositions of the form $(i_{j+1}\, i_j)$ in
$\precj$.  Some factors, though potentially none, to its left are of the form
$(a\, i_j)$, where $a \prec i_j$ (thus also $a \prec i_{j+1}$).  We take all these transpositions and reverse the procedure (looking
at \eqref{eq:stage2} as a guide).  We then continue with the new leftmost factor
$(i_{j+1}\, i_{j})$ and repeat.

The final output at the end of Stage 2 is clearly a factorisation in $\sethurprec{\omega}{\precj}$.
The brief arguments above explain how Stage 1 and Stage 2 are both reversible.
It follows that $\Lambda_j$ is invertible, so $\Lambda_j$ is bijective.
\end{proof}

We state two important corollaries of Theorem \ref{thm:bijectioninc}.

\begin{corollary}\label{cor:moncentral}  Let $\prec$ be an ordering on $[n]$.
	Then for each nonnegative integer $g$, 
	there is a bijection $\Lambda^\prec : \sethurprec{\omega}{\prec}
	\rightarrow \sethur{\omega}$.
\end{corollary}
\begin{proof}
	It is well known that any permutation can be written as a product of
	simple reflections (transpositions of the form $(j\, j+1)$).  Fix some such
	product for the permutation defined by
	$\prec$.  Then, for each transposition $(j\, j+1)$ in the product,
	determine the bijection $\Lambda_j$ and compose them all in the order
	of the product.  This is a bijection 
	from $\sethurprec{\omega}{\prec}$ to $\sethur{\omega}$.
\end{proof}
We now give a bijection showing monotone Hurwitz factorisations are central.
\begin{corollary}\label{cor:order}
	Let $\lambda \vdash n$ and suppose that $\omega, \gamma \in C_\lambda$.
	Then for each nonnegative integer $g$, there exists a bijection $\Delta : \sethur{\omega} \rightarrow \sethur{\gamma}$.
\end{corollary}
\begin{proof}
	Since $\omega$ and $\gamma$ are conjugate, there is a $\delta \in
	\Sym{n}$ such that $\gamma = \delta \omega \delta^{-1}$.  Let $\prec$ be the ordering
	$\delta^{-1}(1) \prec \delta^{-1}(2) \prec \cdots \prec \delta^{-1}(n)$.  Then the
	function $\Phi$ that takes an $f \in \sethur{\omega}$ to the
	factorisation $f'$, where $f'$ is obtained from $f$ by replacing every
	transposition $\tau$ by the transposition $\tau^\delta$, is a bijection from $\sethur{\omega}$ to
	$\sethurprec{\gamma}{\prec}$.  However, from Corollary \ref{cor:moncentral},
	there is a bijection $\Lambda^\prec : \sethurprec{\gamma}{\prec} \rightarrow
	\sethur{\gamma}$.  Thus the composition $\Delta : = \Lambda^\prec \circ \Phi$ is the
	requisite bijection.
\end{proof}
Above we proved Corollary \ref{cor:order} (centrality of monotone Hurwitz
factorisations) from Corollary \ref{cor:moncentral},
but it is easy to show that Corollary \ref{cor:order} also implies Corollary
\ref{cor:moncentral}.  Thus the condition in Corollary \ref{cor:moncentral} is
equivalent to the centrality of monotone Hurwitz factorisations.

We now give a combinatorial proof of the centrality of monotone double Hurwitz
factorisations.
\begin{corollary}\label{cor:mondouble}
	Let $\lambda \vdash n$ and suppose $\omega, \gamma \in C_\lambda$.  
	Then for each nonnegative integer $g$, there exists a bijection $\Theta :
	\setdoublehur{\omega} \rightarrow \setdoublehur{\gamma}$.
\end{corollary}

\begin{proof}	
	Since $\omega$ and $\gamma$ are conjugate, there exists a permutation
	$\delta$ such that $\gamma = \delta \omega \delta^{-1}$.  Let $\prec$ be the ordering given by $\delta^{-1}(1) \prec \cdots \prec \delta^{-1}(n)$.  We describe the
	requisite function $\Theta : \setdoublehur{\omega} \rightarrow \setdoublehur{\gamma}$ as
	follows.  Let $f=\sigma \tau_1 \cdots \tau_m \in \setdoublehur{\omega}$ be a
	factorisation of $\omega$.  It
	follow that $\sigma^\delta \tau_1^\delta \cdots \tau_m^\delta = \gamma$, but
	the transpositions $\tau_1^\delta, \ldots, \tau_m^\delta$ do not necessarily satisfy the
	monotone criteria.  Set $\beta: = (\sigma^\delta)^{-1} \gamma$, so
	$\tau_1^\delta \cdots \tau_m^\delta = \beta$.  Then, as in the proof of
	Corollary \ref{cor:order}, we see that $\tau_1^\delta, \ldots,
	\tau_m^\delta$ is a monotone factorisation of $\beta$ relative to $\prec$.
	The genus $g'$ of this factorisation is given by $m = n-c(\beta) + 2g'$, so $\tau_1^\delta, \ldots,
	\tau_m^\delta \in \sethurgprimeprec{\beta}{\prec}$.  Now we apply
	$\Lambda^\prec$ from Corollary \ref{cor:moncentral} to $\tau_1^\delta, \ldots, \tau_m^\delta$ to obtain a
	factorisation $\tau_1' \cdots \tau_m' \in \sethurgprime{\beta}$.
	Finally, set $\Theta(f) = \sigma^\delta \tau_1' \cdots \tau_m'$.   It is
	clear that $\Theta$ is well-defined.  We leave it to the reader to verify it
	is a bijection.
\end{proof}

We note that applying the Hurwitz moves in Cases 1-3 preserves transitivity of
factorisations.  This means that the same proof as above shows that
\emph{transitive} monotone factorisations are also central.  That is,
$T_n(h_m(\Upxi_n))$ lies in the centre of the symmetric group algebra
$Z(\GSym{n})$.  In Section \ref{sec:trans} we shall see that this centrality still holds if we replace $h_m$ with any symmetric function.

\section{A bijection between star factorisations and monotone double Hurwitz
factorisations}\label{sec:combstar}

We now give a bijection that proves Theorem \ref{thm:hurwitzrelation}.

\begin{theorem}\label{thm:bijectiongamma} Let $\omega \in \Sym{n}$.   Then for
	each nonnegative integer
	$g$, there exists a bijection $\Gamma :
	A_g(\omega) \rightarrow \setdoublehur{\omega}$.  
\end{theorem}

\begin{proof}
We first describe the function $\Gamma$.   Let $f \in A_g(\omega)$ be given by 
$$f=(a_1 \, n) (a_2 \, n) \cdots (a_{n-1+k} \, n)=\omega,$$  where
$k=c(\omega)-1+2g$.  Underline the first appearance, when reading from left to
right, of each transposition $(i \, n)$ for $i<n$ in this factorisation.  Each
of these transpositions must appear at least once as the factorisation is
transitive.  This gives
\begin{equation}
\begin{aligned}\label{eq:typfact}
	f &= \underline{(i_1 \, n)} (b_1^2 \, n) \cdots (b_{j_2}^2 \, n) \underline{(i_2
	\, n)} (b_1^3 \, n) \cdots (b_{j_3}^3 \, n) \underline{(i_3 \, n)} \cdots
		\underline{(i_{n-1} \, n)}\\&\,\,\,\,\,\,\,\,\, \cdot (b_1^n \, n) \cdots (b_{j_n}^n \, n) = \omega,
\end{aligned}
\end{equation}
     which defines the ordering on $[n]$ given by 
	 \begin{equation}\label{eq:ordering}
		 i_1 \prec i_2 \prec \cdots \prec i_{n-1} \prec i_n= n.  
	\end{equation}
	We explain the notation and properties of $f$ in \eqref{eq:typfact}.
	\begin{itemize}
    \item There are $j_p$ transpositions between the underlined $(i_{p-1} \, n)$
		and underlined $(i_p \, n)$, where $j_p \geq 0$.  These transpositions are $(b_1^p \, n) (b_2^p \, n) \dots (b_{j_p}^p \, n)$.
    \item There are $j_n$ transpositions after  $\underline{(i_{n-1} \, n)}$,
		and there are no transpositions before $\underline{(i_1 \, n)}$, so  $j_1=0$.
	\item The underlined $(i_p\, n)$ occurs in position $\sum_{i=1}^p (j_i + 1)$.
    \item For each $p = 1,2,\dots,n-1$, there are no factors before
		$\underline{(i_p \, n)}$ containing $i_p$.  Whence,  for each $1 \leq
		t \leq p$ and $s = 1,2,\dots,j_t$, we have $b_{s}^t \prec
		i_p$.
\end{itemize}    

We now shuffle the underlined transpositions in the factorisation using LHM as
follows. Beginning with $p = 2$, we apply a LHM to the pair $\tau \underline{(i_p \, n)}$,
to obtain $\underline{(i_p \, n)} \tau^{(i_p\, n)}$.  Repeatedly apply LHMs to
$\underline{(i_p\, n)}$ and the transposition to its left until the transposition to the left of $\underline{(i_p \, n)}$ is
$\underline{(i_{p-1} \, n)}$.  Repeat this for $p=3, \ldots, n$.  Once 
completed, the first $n-1$ factors in the product will
be $\underline{(i_1 \, n)} \dots \underline{(i_{n-1} \, n)}$.  Also, since
Hurwitz moves preserve the product of a factorisation, the product of the
factorisation remains $\omega$.

More precisely, when we apply the LHM to the pair $(b_s^p \, n) \underline{(i_p \, n)}$, we obtain
$\underline{(i_p \, n)} (b_s^p \, n)^{(i_p\, n)} = \underline{(i_p \, n)} (b_s^p
\, i_p)$.  As noted in the fourth bullet point,
we have $b_s^p \prec i_p$, so both transpositions are written consistently with
$\prec$.  Any Hurwitz move involving $(b_s^p \, i_p)$ at a
later step of this process will be with some $\underline{(i_{p'} \, n)}$ for $p'
> p$.  Since $b_s^p \prec i_p \prec i_{p'}$, the transpositions $(b_s^p \, i_p)$
and $\underline{(i_{p'} \, n)}$ commute, so both will be unchanged by the
Hurwitz move.  This means that the transposition $(b_s^p \, i_p)$ appears in the factorisation outputted at the end of the process.

    Therefore, this process gives the factorisation
	\begin{equation}\label{eq:continue} 
        \underline{(i_1 \, n)} \dots \underline{(i_{n-1} \, n)}  (b_{1}^2  \, i_2) \dots (b_{j_2}^2 \, i_2) \dots (b_1^n  \, n) \dots (b_{j_n}^n \, n)  = \omega
    \end{equation} 
	There are $k = c(\omega) - 1 +2g$ transpositions in \eqref{eq:continue}, and all are, crucially, in monotone order
	relative to $\prec$.
	    Observe that $(i_1 \, n) \dots (i_{n-1} \, n) = (i_1\, i_2\, \cdots\, i_{n-1}
    \, n)$ and $i_n = n$, so \eqref{eq:continue} becomes 
\begin{equation}\label{eq:continue1} 
        \underline{(i_1\, i_2\, \cdots\, i_{n-1} \, i_n)}  (b_{1}^2  \,
	i_2) \dots (b_{j_2}^2 \, i_2) \dots (b_1^n  \, i_n) \dots (b_{j_n}^n \,
	i_n)  = \omega.
\end{equation} 
From \eqref{eq:continue1} the process is reversible.  That is, given
a full cycle $\sigma = (i_1\, \cdots \, i_n)$ and a factorisation $\sigma \tau_1
\cdots \tau_k$ of $\omega$, where $\tau_1, \ldots,
\tau_k$ are a sequence of transpositions ordered consistently with $i_1 \prec
\cdots \prec i_n = n$, we can reconstruct the factorisation in
\eqref{eq:continue}.   Then to each underlined $(i_j \, n)$, in the
order $j=n-1,n-2,\dots,1$, apply RHM repeatedly to $(i_{j} \, n)$ and the transposition
to its right until $(i_{j} \, n)$ lies to the left of a
transposition of the form $(a\, n)$.  This produces a star factorisation.

Call the portion of \eqref{eq:continue1} that is not underlined (the sequence of
transpositions) $h$ and the portion that is underlined (the cycle) $\sigma$.
Note that product of the transpositions of $h$ equals $\gamma := \sigma^{-1} \omega$.  While the number of
transpositions in $h$ is fixed (it is $k = c(\omega) - 1 + 2g$), the genus of $h$ as a
monotone (relative to $\prec$) factorisation of $\gamma$ is not fixed:  its genus $g'$ satisfies $k =
n-c(\gamma) + 2g'$.   Thus $h \in \sethurgprimeprec{\gamma}{\prec}$, with the ordering $\prec$ defined by the cycle $\sigma$,
\emph{i.e.} $i_1 \prec i_2 \prec \cdots \prec i_n = n$.   
The remainder of the proof uses the bijection $\Lambda^\prec$ in
Corollary \ref{cor:moncentral}.  Through
$\Lambda^\prec$ we obtain a factorisation $h':= \Lambda^\prec (h) \in
\sethurgprime{\gamma}$.  Replacing $h$ in \eqref{eq:continue1} with $h'$
produces a factorisation in $\setdoublehur{\omega}$.   This last step is reversible.  That
is, given a full cycle $\sigma$ and a sequence of transpositions $\tau_1,
\ldots, \tau_k$ such that $f=\sigma  \tau_1 \cdots \tau_k \in
\setdoublehur{\omega}$, write
$\sigma$ as $(i_1\, \cdots\, i_n)$, where $i_n = n$, and consider the ordering
$\prec$ defined by $\sigma$.  We can then
use the bijection $\Lambda^\prec$ between $\sethurgprimeprec{\gamma}{\prec}$ and
$\sethurgprime{\gamma}$ on $\tau_1 \cdots \tau_k$, where $\gamma:= \sigma^{-1} \omega$ and $g'$ is the
appropriate genus given $\gamma$ and the number of transpositions, to produce a factorisation
$f'= \sigma \tau'_1 \cdots \tau'_k = \omega$, where the transpositions are monotone with respect to
 the ordering given by $\prec$.  The factorisation $f'$ is then in the form
 \eqref{eq:continue1}.  As noted above, from \eqref{eq:continue1} we can recover
 a star factorisation.  Thus $\Gamma$ is bijective.
\end{proof}

The centrality of star factorisations now follows from \eqref{eq:centstar}.

There is another characterisation of the centrality of star factorisations.   For an $i \in
[n]$, let
$A_g^i(\omega)$ be defined the same as $A_g(\omega)$ except that the distinguished symbol that occurs in
every factor is $i$ instead of $n$.   Thus
$A_g(\omega) = A_g^n(\omega)$ by definition. Through conjugation, centrality can then be characterized as
$|A_g(\omega)| = |A_g^i(\omega)|$ for any $\omega$ and $i \in [n]$.   We call
the members of $A_g^i(\omega)$ star factorisations \emph{with root
$i$}.

The bijection $\Gamma$ can be used to construct a bijection
between $A_g(\omega)$ and $A_g^i(\omega)$ for any fixed $i$.  Let $f \in
A_g(\omega)$.  Suppose that $\Gamma(f) = \sigma  \tau_1 \cdots \tau_k \in
\setdoublehur{\omega}$.  Now write the cycle $\sigma = (i_1\, \cdots\, i_n)$
with $i_n$ equal to $i$ instead of $n$,  and consider the ordering $\prec$ given by $i_1 \prec \cdots
\prec i_n$.  Set $\gamma:= \sigma^{-1} \omega$.  Then $\tau_1 \cdots \tau_k \in
M_{g'}(\gamma)$ for some $g'$.  Now use the inverse of the bijection
$\Lambda^\prec : \sethurgprimeprec{\gamma}{\prec}$ to
$\sethurgprime{\gamma}$ to produce a factorisation $\sigma  \tau'_1
\cdots \tau'_k$ of $\omega$, where the transpositions are monotone with respect to $\prec$.
This gives us a factorisation as in \eqref{eq:continue1}, and we reverse the
steps prior to \eqref{eq:continue1} to produce a star factorisation with root
$i_n = i$.
\section{The transitivity operator on symmetric polynomials of Jucys-Murphy
elements}\label{sec:trans}

We now give a proof of Theorem \ref{thm:transsym}.
\begin{proof} [Proof of Theorem \ref{thm:transsym}]
    We fix a positive integer $n$ throughout and use the abbreviation $T: =
    T_n$.  We show that the result is true for the elementary symmetric functions
    $e_\lambda$, where $\lambda$ is some partition ($\lambda$ does not need to
	be a partition of $n$).  Since these symmetric functions form a linear basis for $\Lambda$, this proves the theorem.  Note that it is not
    enough to show the property for $e_k$ (\emph{i.e.} the one part elementary
    symmetric functions) since $T(e_k)T(e_j) \neq T(e_k e_j)$.    
We start by analysing $e_k(\Upxi_n)$.

By definition,
\begin{equation*}
	e_k(\Upxi_n) = \sum_{1 \leq i_1 < \cdots < i_k \leq n} J_{i_1} \cdots J_{i_k}.
\end{equation*}
Recall that $J_{t} = \sum_{j=1}^{t-1} (j\, t)$.  It follows that a term $\tau_1
\cdots \tau_k$ from $J_{i_1} \cdots J_{i_k}$ has the form 
\begin{equation}\label{eq:unique}
	\tau_1 \cdots \tau_k = (j_1\, i_1)(j_2\, i_2) \cdots (j_k\, i_k), 
\end{equation}
where $j_t < i_t$ and $i_t < i_{t+1}$ for all $t$;  that is, the
transpositions constitute a \emph{strictly} monotonic factorisation of its
product $\omega$, which has $n-k$ cycles.  Conversely, from Theorem \ref{thm:jm},
each $\omega \in \Sym{n}$ with $c(\omega) = n-k$ is expressed in this sum precisely once, so
it has a unique expression of the form in \eqref{eq:unique}.  Furthermore, in
this unique expression of $\omega$, since $c(\omega) =
n-k$, a standard join-cut analysis (see Section \ref{sec:recurs}) gives that no transposition in
\eqref{eq:unique} is a cut, whence every transposition in \eqref{eq:unique} is
contained in the support of some cycle of $\omega$.  Thus it follows that
$\Orb(\tau_1, \ldots, \tau_k) = \Orb(\omega)$.

Now let $m$ be a positive integer and $\lambda = (\lambda_1, \dots, \lambda_\ell) \vdash m$.  Extending
Theorem \ref{thm:jm} to $\lambda$ gives
    \begin{equation}  \label{eq:elem}
        e_\lambda(\Upxi_n) = \prod_{i=1}^\ell \sum_{\substack{\omega \\ c(\omega) = n-\lambda_i}} \omega.
    \end{equation}
    
    Therefore applying the transitivity operator to \eqref{eq:elem} gives
    \begin{align*}
       T( e_\lambda(\Upxi_n) ) = \sum_{(\omega_1, \dots, \omega_\ell)}  \omega_1
	   \dots \omega_\ell,
    \end{align*}
	where the sum is over all $(\omega_1,  \dots, \omega_\ell)$ with
    $c(\omega_i) = n - \lambda_i$ for each $i$, and where $\Orb(\omega_1, \dots,
    \omega_\ell) = \{[n]\}$.

    Now take any $\tau \in \Sym{n}$.  In order to complete the proof we must show that
	\begin{equation}\label{eq:toprove} 
        \tau T( e_\lambda(\Upxi_n) ) \tau^{-1} = T( e_\lambda(\Upxi_n) ).
    \end{equation} 
    However, for any $\omega_1 \omega_2 \dots \omega_\ell$, we have
    $$
        \tau \omega_1 \omega_2 \dots \omega_\ell \tau^{-1} = \tau \omega_1
		\tau^{-1} \tau \omega_2 \tau^{-1} \dots \tau \omega_\ell \tau^{-1}.
    $$
    Each $\tau \omega_i \tau^{-1}$ satisfies $c(\tau \omega_i \tau^{-1}) =
    c(\omega_i) = n - \lambda_i$.  Also, since conjugating by $\tau$ just
    permutes the symbols in each factor, the set $\{\omega_1, \omega_2, \dots,
	\omega_\ell\}$
    acts transitively on $[n]$ if and only if the set $\{\tau \omega_1 \tau^{-1}, \tau
	\omega_2 \tau^{-1}, \dots, \tau \omega_\ell \tau^{-1}\}$ acts transitively on
	$[n]$.  Thus each term on the left-hand side of \eqref{eq:toprove}
	corresponds to a term on the right-hand side, completing the proof.
\end{proof}

\section{Open problems}\label{sec:open}

Recall that Theorem \ref{thm:jucyscentral} also has a converse: that the algebra
generated by $f(\Upxi_n)$ over all symmetric functions $f$ is exactly
$Z(\GSym{n})$.  It is not immediately clear to us whether the algebra generated
by the transitive images of $f(\Upxi_n)$ over all symmetric functions $f$ is $\ZSym{n}$.

\begin{question}
	What is the subalgebra of $Z(\GSym{n})$ generated by $T_n(f(\Upxi_n))$ over all symmetric functions $f$?
\end{question}

Theorem \ref{thm:transpower} gives an expression for $T_n(p_k(\Upxi_n))$, and
the proof of Theorem \ref{thm:transsym} considers $T_n(e_\lambda(\Upxi_n))$.  
\begin{question}
	Are there any other symmetric functions $f$ for which $T_n(f(\Upxi_n))$ has
	a nice expression in terms of other symmetric functions evaluated at the
	Jucys-Murphy elements?
\end{question}
Goulden and Jackson \cite{goul:8} give formulae calculating the numbers
$a_g(\lambda)$ in terms of other factorisations in $\Sym{n}$, so through
Theorem \ref{thm:hurwitzrelation} these likewise give formulae for monotone
double Hurwitz numbers.  We 
work with an equivalent expression for $a_g(\lambda)$ given by F\'eray
\cite{feray4}:  if $\lambda = (\lambda_1, \lambda_2, \ldots)$, the author found
\begin{equation}\label{eq:feray}
	a_g(\lambda) = \frac{(2g + n + \ell(\lambda) -2)!}{n!} \left( \prod_i
		\lambda_i \right)
	[t^{2g}] f(t)^{n-2} \prod_i f(\lambda_i t),
\end{equation}
where $f(t) = 2t^{-1} \sinh\left(\frac{t}{2}\right)$ and $[t^{2g}]$ extracts the
coefficient of $t^{2g}$ in the series that follows.  We obtain two 
expressions for monotone double Hurwitz factorisations in the case $\lambda = (n)$
and $(1^n)$ that have combinatorial significance.

In the case $\lambda = (n)$, we obtain from \eqref{eq:feray}
	\begin{equation*}
		a_g((n)) = \frac{(2g + n -1)!}{n!} 
		[t^{2g}] \left( e^{n \frac{t}{2}} \frac{(e^{\frac{t}{2}} -
			e^{-\frac{t}{2}})^{n-2}}{t^{n-1}}  - e^{-n\frac{t}{2}}
		\frac{(e^{\frac{t}{2}} - e^{-\frac{t}{2}})^{n-2}}{t^{n-1}} \right).
	\end{equation*}		
	Let $h(t)$ be the series following $[t^{2g}]$ in the previous equation.
	Observe that $h(t)$ can be
	expressed as $h(t) = g(t) + g(-t)$, where
	\begin{equation*}
		g(t) = e^{n \frac{t}{2}} \frac{(e^{\frac{t}{2}} -
	e^{-\frac{t}{2}})^{n-2}}{t^{n-1}},
	\end{equation*}
	whence $[t^{2g}] h(t) = 2[t^{2g}] g(t)$.  Applying this
	observation, we obtain
	\begin{align}
		a_g((n)) &= \frac{(2g + n -1)!}{n!} 2[t^{2g + n -1}]  e^t (e^t
		-1)^{n-2}\notag\\ &= 
		\frac{(2g + n -1)!}{n!} 2 [t^{2g + n -1}] \frac{\mathrm{d}}{\mathrm{d} t}
		\frac{(e^t-1)^{n-1}}{n-1}\notag\\
			    &= \frac{(2g+n)!}{\binom{n}{2}} [t^{2g+n}]
			    \frac{(e^t-1)^{n-1}}{(n-1)!}\notag\\
		 &= \frac{S(2g+n, n-1)}{\binom{n}{2}}\label{for:questiona},
	\end{align}
	where $S(\cdot,\cdot)$ is the \emph{Stirling number of the second kind};  that is,
$S(2g+n,n-1)$ is the number of ways to partition a set of $2g+n$ objects into
$n-1$ nonempty parts.  The last equality \eqref{for:questiona} follows from a
well-known generating series for Stirling numbers;  see \cite[Section 3.3]{aigner}.

When $\lambda = (1^n)$,  we obtain from \eqref{eq:feray}
\begin{align} 
	a_g((1^n)) &= \frac{(2g+2n - 2)!}{n!} [t^{2g}] \left(2t^{-1} \sinh
	\left( \frac{t}{2} \right) \right)^{2n-2}\notag\\
			   &=\frac{(2g+2n - 2)!(2n-2)!}{n!} [t^{2g + 2n -2}] \frac{1}{(2n-2)!} \left(2
			   \sinh\left(\frac{t}{2}\right) \right)^{2n-2}\notag\\
			   &= (n-1)! \mathrm{Cat}(n-1) T(g+n-1, n-1)\label{for:questionb},
\end{align}
where $T(\cdot, \cdot)$ is the \emph{central factorial number} and
$\mathrm{Cat}(n) = \tfrac{1}{n+1}\binom{2n}{n}$ is the $n^\upth$ Catalan number.  The last equality
\eqref{for:questionb} follows from \cite[Exercise 5.8]{StanEC2},
where the author also explains that $T(g+n-1, n-1)$ counts the number
of ways to partition $\{1,2,\dots,g+n-1, 1',2',\dots,(g+n-1)'\}$ into $n-1$ nonempty parts so that for each part $P$, if $i$ is the least integer so that $i$ or $i'$ is in $P$, then both $i$ and $i'$ are in $P$. 

Thus from \eqref{for:questiona}, \eqref{for:questionb} and Theorem
\ref{thm:hurwitzrelation}, we obtain the following expressions for
monotone double Hurwitz factorisations:
\begin{equation}
\begin{aligned}\label{eq:formulae}
	\doublehur{(n)} &= \frac{S(2g+n, n-1)}{\binom{n}{2}}, \\
	\textnormal{ and }\,\, \doublehur{(1^n)} &= (n-1)! \mathrm{Cat}(n-1) T(g+n-1, n-1)
\end{aligned}
\end{equation}
The first formula in \eqref{eq:formulae} is new; the second is essentially
\cite[Equation (23)]{jucys_matrix_int-matsumoto-novak}.  There, the authors find
the special case $\hur{(n)}$, which is equivalent to finding
$\doublehur{(1^n)}/(n-1)!$, using techniques substantially different from ours.
The second formula in \eqref{eq:formulae}, along with a recurrence for $T(\cdot,
\cdot)$ given in \cite[Exercise 5.8]{StanEC2}, gives rise to the recurrence
\begin{equation} \label{for:questionc}
	n\cdot \doublehur{(1^n)} = n(n-1)^2 md_{g-1}((1^n)) + 2(n-1)(2n-3)
	\doublehur{(1^{n-1})}.
\end{equation}
Note the last term in \eqref{for:questionc} pertains to factorisations in
$\Sym{n-1}$.
This recurrence is reminiscent of the Harer-Zagier recurrence for one face maps,
and suggests a simple combinatorial interpretation of this type of monotone
double Hurwitz factorisations may exist.  Finally, Goulden and Jackson \cite{goul:8} give a
simple relation between the number of transitive star factorisations and a
special case of the double Hurwitz factorisations.  We use the notation of
Goulden and Jackson, where they set $b_g(\alpha) := \tfrac{|H^g_{(n),
\alpha}|}{|C_\alpha|}$ for $\alpha \vdash n$, where $H^g_{(n), \alpha}$ is
defined in Section \ref{sec:connect}.
Then
\begin{equation} \label{for:questiond}
	\begin{aligned}
		b_g(\alpha \cup 1^{n-1})&= n! (2n-1)^{n + \ell(\alpha) + 2g - 3}
	a_g(\alpha)\\
									 &= n! (2n-1)^{n + \ell(\alpha) + 2g - 3}
	\doublehur{\alpha}.
\end{aligned}
\end{equation}
The partition $\alpha \cup 1^{n-1}$ has all the parts of $\alpha$ with an
additional $n-1$ parts equal to 1;  thus
$b_g(\alpha \cup 1^{n-1})$ counts factorisations in $\Sym{2n-1}$.
The first equality is from \cite[Corollary 1.4]{goul:8},
while the second follows from Theorem \ref{thm:hurwitzrelation}, so is new. 
In fact, Goulden and Jackson suggested proving the
first equation of \eqref{for:questiond} combinatorially to prove the centrality
of star factorisations.  We did not use their suggested method, so finding a
combinatorial proof for the first equation in \eqref{for:questiond} remains
open.
Intriguingly, the last equality connects double Hurwitz factorisations in
$\Sym{2n-1}$ with monotone double Hurwitz factorisations in $\Sym{n}$.
\begin{question}
	Are there  combinatorial explanations of formulae \eqref{for:questiona},
	\eqref{for:questionb}, \eqref{eq:formulae} or \eqref{for:questionc}?  Are there combinatorial 
	explanations connecting the left-hand side of \eqref{for:questiond} with either right-hand
	side?
\end{question}

\bibliographystyle{alpha}
\bibliography{star_factorisation_symmetry}

\end{document}